\documentclass{amsart}
\usepackage{amssymb}
\usepackage{amsmath}
\usepackage{amsfonts}
\newtheorem{theorem}{Theorem}[section]
\newtheorem{lemma}[theorem]{Lemma}

\newtheorem{prop}[theorem]{Proposition}

\newtheorem{definition}[theorem]{Definition}
\newtheorem{remark}[theorem]{Remark}

\newcommand{\tS}{{\tilde S}}

\begin{document}
\title{A continuity argument for a semilinear Skyrme model}

\author{Dan-Andrei Geba and S. G.  Rajeev }

\address{Department of Mathematics, University of Rochester, Rochester, NY 14627}
\email{dangeba@math.rochester.edu}
\address{Department of Physics and Astronomy, Department of Mathematics, University of Rochester, Rochester, NY 14627}
\email{rajeev@pas.rochester.edu}
\date{}

\begin{abstract}
We investigate a semilinear modification for the wave map problem
proposed by Adkins and Nappi \cite{AN}, and prove that in the equivariant case the solution remain
continuous at the first possible singularity. This is usually one
of the steps in proving existence of global smooth solutions for
certain equations.
\end{abstract}


\maketitle
\section{Introduction}

Let $\phi :\mathbb{R}^{n+1}\to \it{M}$ be a map from the $n+1$
dimensional spacetime, with Lorentzian metric $g$ of signature
$(1,n)$,  to a Riemannian manifold $(M,h)$. The action of the wave
map equation, or the nonlinear $\sigma$ model in physics
terminology, is
\begin{equation}
S\,=\,\frac 12 \int  g^{\mu\nu} \,\partial_\mu\phi^i\, \partial_\nu\phi^j\,
h_{ij}(\phi)\,dg \label{awm}
\end{equation}

The initial value problem for the Euler-Lagrange equations
associated with \eqref{awm} has been intensely studied, especially
the issues of global existence and regularity for its solutions. We
mention here the pioneering works of Christodoulou -
Tahvildar-Zadeh \cite{MR1223662}, Grillakis \cite{G}, Shatah -
Tahvildar-Zadeh \cite{MR1278351}, and Struwe \cite{MR1990477}.

The particular case when $M= \mathbb{S}^3$ and  $n=3$ is of
special interest in high energy physics. The nuclei of atoms are
held together by forces mediated by the pi mesons. These are a set
of three  particles  whose masses are small compared to the nuclei
themselves, so to a first approximation they can be considered to
be massless, i.e., travelling at the speed of light. If
interactions among them are ignored, the pi mesons are described
by a field $\phi:R^{3+1}\to R^3$ satisfying the wave equation.
Interactions would add nonlinearities. A remarkable fact of
physics is that the interactions among the pi mesons are
described, to a good approximation (Gell-Mann - Levy \cite{MR0140316}, Gursey \cite{MR0118418},
\cite{MR0159685}, and Lee
\cite{MR0446201}), by considering the target manifold to be the
sphere $ \mathbb{S}^3$ and replacing the wave equation by the
corresponding wave map equation. Physicists call this the
nonlinear $\sigma$ model for historical reasons.

In order for the energy to be finite, the gradient of the field
must vanish at infinity: it tends to the same value in all
directions of spatial infinity. Thus, each instantaneous pi meson
configuration corresponds to a map $\tilde\phi:\mathbb{S}^3\to
\mathbb{S}^3$ obtained by identifying the points at spatial
infinity. A continuous map of the sphere to itself has an integer
associated with it, the winding number or element of the homotopy
group $\pi_3(\mathbb{S}^3)$. In terms of the original field, we
can write this as

\begin{equation}
Q= c \int \epsilon_{ijk}\partial_a\phi^i\partial_b\phi^j \partial_c\phi^k \epsilon^{abc}\,dx
\label{Q}
\end{equation}
where $\epsilon$ is the Levi-Civita symbol and $c$ is a normalizing constant. Small perturbations
from the constant configurations would have $Q=0$. If time
evolutions were continuous, $Q$ would be a conserved quantity, a
`topological charge'.

It was suggested by Skyrme \cite{MR1305831} in the 1960s that this
topological charge $Q$ is just the total number of neutrons and
protons in a nucleus: the baryon number. Thus, a proton (hydrogen
nucleus) would have $Q=1$, a helium nucleus would have either
$Q=3$ or $4$ (depending on the isotope), and so on. This has been
confirmed by connecting with Quantum chromodynamics (Balachandran
- Nair - Rajeev - Stern \cite{BNRS2}, \cite{BNRS}, Witten
\cite{MR717916}, \cite{MR717915}), the fundamental theory of
nuclear interactions. Numerical calculations (Battye - Sutcliffe
\cite{PhysRevLett.79.363}) also support this idea. Some essential
theoretical puzzles had to be resolved before this rather strange
idea of Skyrme could be established. These were resolved in the
mid-80's; e.g.,  by using topological notions (`anomalies')
(\cite{MR1305831}, \cite{BNRS2}, \cite{BNRS}, \cite{MR717916},
\cite{MR717915}).

One of these puzzles is that actually the wave map 
does {\em not} have continuous time evolution: it is
supercritical. Shatah \cite{MR933231} has exhibited an example of
a solution with smooth initial conditions that breaks down in finite time. Physically, this is because the
forces among the pi mesons are mainly attractive; so a
configuration of winding number one would shrink to a point,
emitting its energy as radiation carried away to infinity. At the
singularity, $Q$ would jump to zero. Since baryon number is
strictly conserved by nuclear forces, this cannot be right.

Skyrme suggested a modification of the action of the wave map that
could stabilize configurations with $Q\neq 0$. The
corresponding equations are quite intricate, being quasilinear.
What we study here is a semilinear modification of the above
action $S$, previously introduced by Adkins - Nappi
\cite{AN}, which is also expected, for physical reasons, to have
regular solutions\footnote{We consider only the case when the masses of the mesons are set to zero, which is
sufficient to understand short distance singularities.}. The idea
is to add  a short range repulsion among the pi mesons, created by
their interactions with an omega meson. Geometrically, it
describes the nonlinear coupling of the wave map with a gauge
field, the source (charge density) of the gauge field being the
density of the topological charge $Q$.

In detail, the action of the theory we investigate is

\begin{equation}
\tS\,=\,S\, +\, \frac{1}{4} \int  F^{\mu\nu} F_{\mu\nu}\,dg\,-\, \int A_\mu j^\mu \, dx
\label{S}
\end{equation}
where the omega meson is represented by the gauge potential $A=A_\mu dx^\mu$, the 2-form
$F_{\mu\nu}\,=\,\partial_\mu A_\nu - \partial_\nu A_\mu$ is its
associated electromagnetic field, and
\[
j^\mu\,=\,c \,\epsilon^{\mu\nu\rho\sigma}
\partial_\nu\phi^i \partial_\rho\phi^j \partial_\sigma\phi^k
\epsilon_{ijk}
\]
is the flux or the baryonic current\footnote{Note that $Q=\int j^0 \,dx$.}. 

Here we study time dependent equivariant maps associated
to $\tS$ with winding number 1, i.e.,
\[
g= \text{diag}(-1,1,1,1)\qquad M=\mathbb{S}^3
\]
\[
\phi(t,r,\psi, \theta)= (u(t,r),\psi, \theta)\qquad A(t,r,\psi,
\theta)=(V(t,r),0,0,0)
\]
with boundary conditions  $u(t,0)=0$, $u(t,\infty)= \pi$, and $V(t,\infty)= 0$.

Routine computations lead to
\[\aligned
\frac{\tS}{2\pi}\,=\,\int & \left[ r^2(-u_t^2+u_r^2)
+ 2\sin^2u+ \frac{\alpha^2 (u-\sin u \cos u)^2}{r^2}\right]\,dr\, dt -\\ 
&- \int \left[r V_r+\frac{\alpha (u-\sin u \cos
u)}{r}\right]^2  \,dr\, dt
\endaligned
\]
where $\alpha$ is an appropriate constant. Eliminating $V$ using its variational equation and 
scaling the coordinates by a factor of $\alpha$, we obtain the main equation satisfied by $u$:
\begin{equation}
u_{tt} - u_{rr} - \frac 2r u_r + \frac{1}{r^2} \sin 2u+
\frac{1}{r^4}(u-\sin u \cos u)(1- \cos 2u) = 0\label{main}
\end{equation}

Based on the physical intuition detailed above, \textbf{we
conjecture that smooth finite energy initial data for this equation will evolve into globally
regular solutions}. This is supported also by the existence of
static numerical solutions of winding number 1 for \eqref{S}, which
were found in \cite{AN}. Regularity at the origin forces
the initial data to vanish at $r=0$, which is consistent with the
previously imposed boundary condition, $u(t,0)=0$. 

One needs to compare \eqref{main} with the corresponding wave map
equation
\begin{equation}
u_{tt} - u_{rr} - \frac 2r u_r +  \frac{1}{r^2} \sin 2u =
0\label{wmap}
\end{equation}
for which Shatah's solution (later found by Turok and Spergel
\cite{PhysRevLett.64.2736} in close form)
\[
u(t,r)\,=\,2\arctan{\frac{r}{T_0-t}}
\]
provides an example of smooth initial data that develop
singularities in finite time.

The goal of this article is to take the first step in proving our
conjecture, which is to show that a smooth solution for
\eqref{main} remains continuous at the first possible singularity.
Using the fact that \eqref{main} is semilinear  and it is invariant under translations, we can assume, without any loss of generality,
that our solution starts at time $t= -1$ and breaks down at the origin $(0,0)$. Our main result is

\begin{theorem}
The solution $u$ for \eqref{main} is continuous at the origin and
\[
u(t,r)\to 0 \quad \text{as} \quad (t,r)\to (0,0)
\]\label{cont}
\end{theorem}

The proof follows the lines of the one for $2+1$ dimensional equivariant
wave maps (e.g., \cite{MR1168115}); however, because we are
in $3+1$ dimensions, we lose a
critical estimate which is in turn bypassed by a new argument, using the
sign of the last term in \eqref{main}:
\begin{equation}
u \cdot \frac{1}{r^4}(u-\sin u \cos u)(1- \cos 2u)\geq 0
\label{pos}
\end{equation}

In all of these problems (i.e., wave maps and Skyrme model) the initial configuration shrinks as time evolves, causing energy and winding number to accumulate at the origin. For the $2+1$ dimensional
wave map equation, enough energy is radiated away so that  the energy density at the origin is finite. In $3+1$ dimensions, the  energy density  diverges at the origin although the total energy tends to zero. For the Adkins - Nappi version of the Skyrme model, there is a repulsive part in the energy density. Our new argument shows that the energy corresponding to this part does not concentrate, allowing us to prove the continuity of the field.

\section{Main argument}
We will be working with backward truncated cones, their mantels
and bases, denoted as
$$
\aligned
&K_T^S\,=\,\{(t,r)|\ T\leq t\leq S,\, 0\leq r\leq |t|\}\\
&C_T^S\,=\,\{(t,r)|\ T\leq t\leq S,\, r= |t|\}\\
&B_T\,=\,\{(t,r)|\ t=T,\, 0\leq r\leq |t|\}
\endaligned
$$
where $-1\leq T \leq S \leq 0$. For narrower bases, we use the
notation
\[
B_T(\lambda)\,=\,\{(t,r)|\ t=T,\, -\lambda t\leq r\leq -t\}
\]
where $0\leq\lambda\leq 1$.

Multiplying \eqref{main} by $r^2u_t$, $r^2u_r$, $r^3u_r$, $r^2 u$,
respectively $r^2tu_t$, we obtain the following differential
identities:

\begin{lemma}
A classical solution for \eqref{main} satisfies:
\begin{equation}
\partial_t\left(\frac{r^2}{2}(u_t^2+u_r^2)+\sin^2u+\frac{1}{2r^2}(u-\sin
u \cos u)^2\right)-\partial_r(r^2u_tu_r)=0 \label{e1}
\end{equation}

\begin{equation}
\aligned
\partial_t(r^2u_tu_r)-&\partial_r\left(\frac{r^2}{2}(u_t^2+u_r^2)- \sin^2u-\frac{1}{2r^2}(u-\sin
u \cos u)^2\right)=\\
& = r(u_r^2-u_t^2)- \frac{1}{r^3}(u-\sin u \cos u)^2
\endaligned\label{e2}
\end{equation}

\begin{equation}
\aligned
\partial_t(r^3u_tu_r)-&\partial_r\left(\frac{r^3}{2}(u_t^2+u_r^2)- r\sin^2u-\frac{1}{2r}(u-\sin
u \cos u)^2\right)=\\
&=\frac{r^2}{2}(u_r^2-3u_t^2)+\sin^2u-\frac{1}{2r^2}(u-\sin u \cos
u)^2
\endaligned\label{e3}
\end{equation}

\begin{equation}
\aligned
\partial_t(r^2uu_t)-\partial_r(r^2uu_r)&=r^2(u_t^2-u_r^2)-\\
& -u\left(\sin2u+\frac{1}{r^2}(u-\sin u \cos u)(1-\cos2u)\right)
\endaligned
\label{e4}
\end{equation}

\begin{equation}
\aligned
\partial_t&\left(\frac{tr^2}{2}(u_t^2+u_r^2)+t\sin^2u
+\frac{t}{2r^2}(u-\sin
u \cos u)^2\right)-\partial_r(tr^2u_tu_r)=\\
&=\frac{r^2}{2}(u_t^2+u_r^2)+\sin^2u+\frac{1}{2r^2}(u-\sin u \cos
u)^2
\endaligned\label{e5}
\end{equation}
\end{lemma}

\begin{remark}
The argument for $2+1$ dimensional equivariant wave maps relies on
the counterparts of all five identities. For our analysis we will not
use \eqref{e2}.
\end{remark}

Next, we define the local energy and the flux associated to our
problem:
\begin{definition}
The energy of the time slice $t=T$ is defined as
\[
E(T)\,=\,\int_{B_T}\frac{1}{2}(u_t^2+u_r^2)+\frac{1}{r^2}
\sin^2u+\frac{1}{2r^4}(u-\sin u \cos u)^2
\]
while the flux between the time slices $t=T$ and $t=S$ is given by
\[
F(T,S)\,=\,\frac{1}{\sqrt 2}\int_{C_T^S}
\frac{1}{2}(u_t-u_r)^2+\frac{1}{r^2}\sin^2u+\frac{1}{2r^4}(u-\sin
u \cos u)^2
\]
\end{definition}

The classical energy estimate, obtained by integrating \eqref{e1}
over $K_T^S$,
\[
E(T)-E(S)\,=\,F(T,S)
\]
implies that the local energy is decreasing and the flux decays to
0:
\begin{equation}\label{mone}
E(S)\leq E(T)\quad \text{for}\quad -1\leq T\leq S
\end{equation}
\begin{equation}
\lim_{T\to 0-}F(T,0)\,=\,0 \label{flux}
\end{equation}

These two facts allows us to go further and show that:
\begin{prop}
The solution $u$ is bounded, with
\begin{equation}
\|u\|_{L^\infty}\leq C(E(-1)) \label{li}
\end{equation}
and
\begin{equation}
\lim_{t\to 0-} u(t,-t)=0 \label{f0}
\end{equation}
\label{uli}
\end{prop}
\begin{proof}
For
\[
I(z)=\int_0^z (w-\sin w \cos w)\,dw = \frac{z^2-\sin^2 z}{2}
\]
we have
\[
I(0)=0 \qquad I(z)>0(z\neq 0)\qquad \lim_{|z|\to\infty}
I(z)=\infty
\]
Based on $u(t,0)= 0$ we can write
\[
I(u(t,r))=\int_0^r (u-\sin u \cos u)u_r\,dr
\]
which implies
\[
I(u(t,r))\lesssim \left(\int_{B_t} u_r^2\right)^\frac 12\cdot
\left(\int_{B_t} \frac{(u-\sin u \cos u)^2}{r^4}\right)^\frac 12
\lesssim E(t)
\]
The monotonicity of the energy \eqref{mone} then settles the
first claim.

Next, denoting $v(t)=u(t,-t)$, we use \eqref{flux} to infer
\[
\lim_{T\to 0-}\int_T^0 \frac{(v-\sin v \cos v)^2}{t^2}\ dt = 0
\]
which leads to the existence of a sequence $t_n \to 0$ for which
$v(t_n)\to 0$. For fixed $T$ and large $n$ we obtain
\[
|I(v(T))-I(v(t_n))|\,\lesssim\, \int_T^{t_n}|(v-\sin v \cos
v)\cdot v_s|\,ds\, \lesssim \,F(T,0)
\]
which proves \eqref{f0}.
\end{proof}

We use these results to reduce
the proof of Theorem \ref{cont} to the one of a local energy
estimate.

\begin{theorem}
If the solution $u$ is smooth on $K_{-1}^{t}$ for all $-1<t<0$ and
\begin{equation}
\lim_{T\to 0-}\int_{B_T} \frac{(u-\sin u \cos u)^2}{r^4} \ =0
\label{h2}
\end{equation}
then $u$ is continuous at the origin and
\[
u(t,r)\to 0 \quad \text{as} \quad (t,r)\to (0,0)
\]\label{red}
\end{theorem}
\begin{proof}
We argue as in Proposition \ref{uli} to deduce
\[
\aligned I(u(t,r))\,&=\,I(u(t,-t)) + \int_{-t}^r (u-\sin u \cos
u)(t,r')\cdot u_r(t,r')\,dr'\lesssim\\
&\lesssim I(u(t,-t)) + E(t)^{\frac 12}\cdot\left(\int_{B_t}
\frac{(u-\sin u \cos u)^2}{r^4}\right)^{\frac 12}
\endaligned
\]
which obviously provides the desired conclusion based on
\eqref{mone} and \eqref{f0}.
\end{proof}

\begin{remark}
Theorem \ref{red} is the point where our argument leaves the approach
for equivariant wave maps. There, one relies on the
weaker bound
\[
I(u(t,r))\,\lesssim \,I(u(t,-t)) + E(t)
\]
and proves nonconcentration of the energy (i.e., $E(t)\to 0$
as $t \to 0$). The crucial ingredient for that analysis is
that the annular energy doesn't concentrate,
\[
E_\lambda(t)\,=\,\int_{B_t(\lambda)} e \rightarrow 0
\]
where $e$ is the energy density and $0\leq\lambda\leq 1$. This is
in turn obtained from
\begin{equation}
((rm)_t-(re)_r)^2\lesssim (e-m)(e+m) \label{el}
\end{equation}
where $m$ is the momentum density. We refer the interested
reader to \cite{MR1223662} or \cite{MR1168115} for more details.

In our case, the corresponding estimate should read
\begin{equation}
((r^2m)_t-(r^2e)_r)^2\lesssim r^2(e-m)(e+m)
\label{els}\end{equation}
where
\[
e=\frac{1}{2}(u_t^2+u_r^2)+\frac{1}{r^2}\sin^2u+\frac{1}{2r^4}(u-\sin
u \cos u)^2 \quad \qquad m=u_r u_t
\]

The left hand side in \eqref{els}, which can be obtained through \eqref{e2}, takes the form
\[
\aligned (r^2m)_t-(r^2e)_r\,=\,r(u_r^2&-u_t^2) + \frac{(u-\sin u
\cos u)^2}{r^3}\,-\,2 \sin 2u\cdot u_r \,-\\
&-\,\frac{2(u-\sin u \cos u)(1-\cos 2u)}{r^2}\,u_r
\endaligned
\]
Precisely the last term above fails to obey the bound in \eqref{els}.
This is the reason why we do not use \eqref{e2} in our argument.
\end{remark}

\section{Local energy estimates}
These estimates are deduced by integrating the differential
identities \eqref{e3}-\eqref{e5} on the backward cone $K_T^0$ and
then allow for $T \to 0$. First, we notice that for $f\in
L^\infty$ one obtains
\begin{equation}
\lim_{T\to 0-} \frac{1}{|T|} \int_{K_T^0} \frac{f}{r^2} =
\lim_{T\to 0-} \int_{B_T} \frac{f}{r^2} =0 \label{bd}
\end{equation}
which allows us to ignore certain terms in the analysis.

\begin{lemma}
For $u$ solution of \eqref{main}, smooth on $K_{-1}^{t}$ for all
$-1<t<0$, the following estimates hold:
\begin{align}
&\lim_{T\to 0-} \frac{1}{|T|} \int_{K_T^0}
\left[\frac{1}{2}(3u_t^2-u_r^2)+\frac{1}{2r^4}(u-\sin u \cos u)^2
\right]-
\frac{1}{|T|}\int_{B_T} r u_r u_t  \,=\,0\label{ie3}\\
&\lim_{T\to 0-} \frac{1}{|T|} \int_{K_T^0}
\left[u_r^2-u_t^2+u\cdot\frac{1}{r^4}(u-\sin u \cos
u)(1-\cos2u)\right]\,=\,0\label{ie4}\\
&\lim_{T\to 0-} \frac{1}{|T|} \int_{K_T^0}
\left[-\frac{1}{2}(u_t^2+u_r^2)-\frac{1}{2r^4}(u-\sin u \cos
u)^2\right]+ E(T)\,=\,0\label{ie5}
\end{align}
\end{lemma}

\begin{proof}
We prove only the first estimate, the other two being treated similarly. As mentioned above, we integrate
\eqref{e3} on $K_T^0$ to infer
\[
\aligned \int_{K_T^0} \left[\frac{1}{2}(3u_t^2\, - u_r^2
)-\frac{1}{r^2}\sin^2u+\frac{1}{2r^4}(u-\sin u \cos
u)^2\right]\,=\, \int_{B_T}r u_r u_t \, +\\
+ \int_{C_T^0} \left[ r\cdot
\left(\frac{1}{2}(u_t-u_r)^2+\frac{1}{r^2}\sin^2u+\frac{1}{2r^4}(u-\sin
u \cos u)^2\right)\right]
\endaligned
\]
Using \eqref{bd} we deduce that
\[
\lim_{T\to 0-} \frac{1}{|T|} \int_{K_T^0}
\frac{1}{r^2}\sin^2u\,=\,0
\]
\eqref{ie3} follows immediately as
\[
\lim_{T\to 0-} \frac{1}{|T|} \int_{C_T^0} \left[ r\cdot
\left(\frac{1}{2}(u_t-u_r)^2+\frac{1}{r^2}\sin^2u+\frac{1}{2r^4}(u-\sin
u \cos u)^2\right)\right]\,=\,0
\]
due to \eqref{flux} and that $r\leq |T|$ in $C_T^0$.
\end{proof}

Finally, combining \eqref{ie3}-\eqref{ie5} and relying on
\eqref{pos}, we obtain:
\begin{equation}
\lim_{T\to 0-} E(T) - \frac{1}{|T|}\int_{B_T} r u_r u_t \,=\,0
\label{non}
\end{equation}
Coupling this with
\[\aligned
E(T) - \frac{1}{|T|}\int_{B_T} r u_r u_t &= \int_{B_T}
\left[\frac{1}{4}(1-\frac{r}{|T|})(u_t-u_r)^2\,+\,
\frac{1}{4}(1+\frac{r}{|T|})(u_t+u_r)^2\right]+\\
&+\int_{B_T}\left[\frac{1}{r^2}\sin^2u+\frac{1}{2r^4}(u-\sin u
\cos u)^2\right]
\endaligned\]
we deduce the sufficient condition \eqref{h2} from Theorem \ref{red},
which finishes the argument.

It is worth noting that \eqref{non} yields also
\[
\lim_{T\to 0-}
\int_{B_T}\left[ (1-\frac{r}{|T|})(u_t-u_r)^2+(u_t+u_r)^2 +\frac{1}{r^2}\sin^2u \right]=0
\]

Thus we obtain that the entire energy does not concentrate, except maybe for 
\[
\int_{B_T}\frac{r}{|T|}(u_t-u_r)^2
\] 
This question is addressed in an upcoming article \cite{GR}.

\section*{Acknowledgements}
We would like to thank Manoussos Grillakis and Daniel Tataru for stimulating  discussions in connection with this work. The first author was supported in part by the National Science Foundation Career grant DMS-0747656. The second author was supported in part by the Department of Energy  contract DE-FG02-91ER40685.

\bibliographystyle{amsplain}
\bibliography{anwb}

\end{document}